\documentclass[11 pt]{amsart}
\usepackage[top=3cm, bottom=3cm,left=3cm, right=3cm]{geometry}

\usepackage{amsmath}
\usepackage{amsfonts}
\usepackage{amssymb}

\date{}

\newcommand{\beqa}{\begin{eqnarray*}}
\newcommand{\eeqa}{\end{eqnarray*}}
\newcommand{\beqn}{\begin{eqnarray}}
\newcommand{\eeqn}{\end{eqnarray}}

\newcommand{\iy}{\infty}

\newcommand{\R}{\mathbb R}

\newcommand{\N}{\mathbb N}

\newcommand{\al}{\alpha}
\newcommand{\be}{\beta}

\newcounter{cnt1}
\newcounter{cnt2}
\newcounter{cnt3}
\newcommand{\blr}{\begin{list}{$($\roman{cnt1}$)$}
 {\usecounter{cnt1} \setlength{\topsep}{0pt}
 \setlength{\itemsep}{0pt}}}
\newcommand{\bla}{\begin{list}{$($\alph{cnt2}$)$}
 {\usecounter{cnt2} \setlength{\topsep}{0pt}
 \setlength{\itemsep}{0pt}}}
\newcommand{\bln}{\begin{list}{$($\arabic{cnt3}$)$}
 {\usecounter{cnt3} \setlength{\topsep}{0pt}
 \setlength{\itemsep}{0pt}}}
\newcommand{\el}{\end{list}}

\newtheorem{thm}{Theorem}[section]
\newtheorem{lem}[thm]{Lemma}

\newtheorem{ex}[thm]{Example}

\newtheorem{Def}[thm]{Definition}
\newtheorem{Prop}[thm]{Proposition}
\newtheorem{rem}[thm]{Remark}
\newcommand{\Rem}{\begin{rem} \rm}
\newcommand{\bdfn}{\begin{Def} \rm}
\newcommand{\edfn}{\end{Def}}
\newcommand{\TFAE}{the following assertions are equivalent: }

\newcommand{\ba}{\begin{array}}
\newcommand{\ea}{\end{array}}

\usepackage{tikz}

\tikzstyle{vertex}=[scale=0.9,auto=left,circle,fill=black!10,inner
sep=0.9pt]
\usetikzlibrary {positioning}
\usetikzlibrary{arrows}
\usepackage{graphics}
\usepackage{graphicx}
\usepackage{hyperref}
\usepackage{graphicx,epsfig}
\usepackage{tikz}
\usepackage{caption}
\usepackage{subcaption}
\usepackage{float}
\usetikzlibrary{shapes,arrows}

\begin{document}
\title[Characterizations of almost greedy and partially greedy bases]
{Characterizations of almost greedy and partially greedy bases}

\author[Dilworth]{S. J. Dilworth}
\address[S.J. Dilworth]{Department of Mathematics \\
University of South Carolina \\
Columbia, SC 29208\\
U.S.A., \textit{E-mail~:} \textit{dilworth@math.sc.edu}}

\author[Khurana]{Divya Khurana}
\address[Divya Khurana]{Department of Mathematics\\ The Wiezmann Institute of Science,
Rehovot\\ Israel, \textit{E-mail~:}
\textit{divya.khurana@weizmann.ac.il, divyakhurana11@gmail.com}}

\subjclass[2000]{46B15; 41A65}

\keywords{almost greedy basis, partially greedy basis, reverse
partially greedy basis}

\begin{abstract}
We shall present new characterizations of partially greedy and
almost greedy bases. A new class of basis (which we call reverse
partially greedy basis) arises naturally from these
characterizations of partially greedy bases. We also give
characterizations for $1$-partially greedy and $1$-reverse partially greedy
bases.
\end{abstract}

\thanks {The first author was supported by the National Science Foundation under Grant Number DMS--1361461
and by the Workshop in Analysis and Probability at Texas A\&M
University in 2017. The second author was supported by the Israel
Science Foundation (Grant Number--137/16).}

\maketitle

\section{Introduction}
Let $X$ be a real Banach space with a seminormalized basis $(e_n)$
and biorthogonal functionals $(e_n^*)$. $X$ is allowed to be finite-dimensional, in which case $(e_n)$ is a finite algebraic basis for $X$.
 For any $x\in X$ we denote
$supp(x)=\{n:e_n^*(x)\not=0\}$, $P_A(x)=\sum_{i\in A} e_i^*(x)e_i$.
If $A\subset\N$ then $|A|$ denotes the cardinality of $A$,
$1_A=\sum_{i\in A} e_i$, for $A,B\subset \N$ we write $A<B$ if
$max~A<min~B$.

In \cite{KT} Konyagin and Temlyakov introduced the
\textbf{Thresholding Greedy Algorithm (TGA)} $(G_m)$, where $G_m(x)$
is obtained by taking the largest $m$ coefficients in the series
expansion of $x$. For $x\in X$ let $\Lambda_m(x)$ be a set of any
$m$-indices such that
\begin{align*}
\underset{n\in \Lambda_m(x)}{min}|e_n^*(x)|\geq \underset{n\not\in
\Lambda_m(x)}{max}|e_n^*(x)|.
\end{align*}

The $mth$ greedy approximation is given by
\begin{align*}
G_m(x)=\sum_{n\in \Lambda_m(x)}e_n^*(x)e_{n}.
\end{align*}

Konyagin and Temlyakov in \cite{KT} defined a basis $(e_n)$ to be
\textbf{greedy} with constant $C$ if for all $x\in X$, $m\in \N$
\begin{align}\label{g}
\|x-G_m(x)\| \leq C\sigma_m(x)
\end{align}
where $\sigma_m(x)$ is the error in the best $mth$ term
approximation to $x$ and is given by
\begin{align}\label{map}
 \sigma_m(x)=inf~\left\{ \left
\|x-\underset{i\in A}{\sum}a_ie_i\right\|:~|A|=m,~a_i\in \R,~i\in
A\right\}.
\end{align}
 They proved that a basis is greedy if and only if it is
 unconditional and democratic. Recall that a basis $(e_n)$ of
a Banach space is unconditional if any rearrangement of the series
$x=\underset{n\geq1}{\sum}e_n^*(x)e_n$ converges in norm to $x$. A
basis $(e_n)$ is said to be democratic if there exists a constant
$\Gamma \geq 1$ such that $\|1_A\|\leq \Gamma \|1_B\|$ where $A,B$
are finite subsets of $\N$ and $|A|\leq |B|$.

They also introduced the notion of \textbf{quasi-greedy} basis. A
basis is said to be quasi-greedy with constant $C$ if
$\|G_m(x)\|\leq C\|x\|$ for all $x\in X$ and $m\in \N$. Later, Wojtaszczyk
\cite{W} proved that a basis is quasi-greedy if
$G_m(x)\longrightarrow x$ as $m\longrightarrow \infty$ for all $x\in
X$.

The class of \textbf{almost greedy} basis and \textbf{partially
greedy} basis was considered in \cite{S}. A basis $(e_n)$ is said to
be almost greedy with constant $C$ if for all $x\in X$, $m\in \N$
\begin{align}\label{ag}
\|x-G_m(x)\| \leq C {\stackrel{\sim}\sigma}_m(x)
\end{align}
where ${\stackrel{\sim}\sigma}_m(x)$ is the error in the best
projection of $x$ onto a subset of $(e_n)$ with size at most $m$ and
is given by
\begin{align}\label{mpp}
{\stackrel{\sim}\sigma}_m(x)=inf \{\|x-P_A(x)\|:~|A|\leq m\}.
\end{align}
A basis $(e_n)$ is said to be partially greedy basis with constant
$C$  if  for all $x\in X $ and $m\in \N$ we have
\begin{align*}
\|x-G_m(x)\| \leq C
\left\|\sum_{n=m+1}^{\infty}e_i^*(x)e_i\right\|.
\end{align*}
In \cite{S} the authors proved that  almost greedy
bases  are characterized by quasi-greediness and democracy and partially greedy bases are characterized by
quasi-greediness and the  conservative property.  A basis $(e_n)$  of a
Banach space is said to be conservative if there exists a constant
$\Gamma_c$ such that $\|1_A\|\leq \Gamma_c \|1_B\|$ whenever $A<B$
and $|A|\leq |B|$.

Characterizations of greedy, almost greedy and quasi-greedy bases
are known for the constant $C=1$. Albiac and Wojtaszczyk \cite{AW}  characterized
 $1$-greedy bases. Later, Albiac and Ansorena
\cite{ALq}, \cite{AL} characterized  $1$-quasi-greedy and $1$-almost greedy bases.

The \textbf{Weak Thresholding greedy algorithm (WTGA)} with weakness
parameter $\tau$, where $0<\tau<1$, was considered in \cite{T}. For a given
basis $(e_n)$ of $X$ and $x\in X$ we denote by  $\Lambda_m^{\tau}(x)$
any  set of  $m$-indices such that
\begin{align*}
\underset{n\in \Lambda_m^\tau(x)}{min}|e_n^*(x)|\geq \tau
\underset{n\not\in \Lambda_m^\tau(x)}{max}|e_n^*(x)|.
\end{align*}

The $mth$ weak greedy approximation is given by
\begin{align*}
G_m(x)=\sum_{n\in \Lambda_m^{\tau}(x)} e_n^*(x)e_{n}.
\end{align*}

A more restrictive version of weak greedy algorithm known as the
\textbf{branch greedy algorithm (BGA)} was considered in
\cite{DKSW}.

Throughout the paper we will use the following notations. For any
$x\in X$, a greedy ordering for $x$  is a $1-1$ map $\rho:\N\longrightarrow \N$ such
that $supp(x) \subset \rho(x)$ and if $j<k$, then
$|e_{\rho(j)}^*|\geq |e_{\rho(k)}^*|$.

If $\rho$ is a greedy ordering for $x\in X$,  set
$\Lambda_m(x):=\{\rho(1),\cdots, \rho(m)\}$ ,  $\al_m(x):=
min~ \Lambda_m(x)$, and $\be_m(x):= max~ \Lambda_m(x)$.

We will work with the following weaker versions of $(\ref{map})$ and
$(\ref{mpp})$:

\begin{align}\label{a1}
\sigma_m^L(x)= inf \left\{ \|x-\underset{i\in A}{\sum}a_ie_i\|:~|A|=
m,~A<\al_m(x),~a_i\in \mathbb{R},~i\in A\right\},
\end{align}
\begin{align}\label{a2}
\sigma_m^R(x)= inf \left\{ \|x-\underset{i\in A}{\sum}a_ie_i\|:~|A|=
m,~A>\be_m(x),~a_i\in \mathbb{R},~i\in A\right\},
\end{align}
\begin{align}\label{a3}
{\stackrel{\sim}\sigma}_m^L(x)=inf\{ \| x-P_A (x) \|:~|A|\leq
m,~A<\al_m(x)\},
\end{align}

\begin{align}\label{a4}
{\stackrel{\sim}\sigma}_m^R(x)=inf \{\| x-P_A (x) \|:~|A|\leq
m,~A>\be_m(x)\},
\end{align}
for $~x\in X,~m\in \N$.
\begin{rem}  Note that  $\liminf_m \sigma_m^L(x) \ge (1+K_b)^{-1}\|x\|$ and $\liminf_m  \sigma_m^R(x) \ge K_b^{-1}\|x\|$,
where $K_b$ is the basis constant of $(e_n)$, whereas $\lim_m \sigma_m(x)=\lim_m \tilde \sigma_m(x)=0$.
\end{rem}

In this paper we will prove new characterizations of \textbf{almost
greedy} and \textbf{ partially greedy} bases. We will prove that a basis
is almost greedy or partially greedy if and only if $(\ref{g})$,
$(\ref{ag})$ are satisfied where $\sigma_m(x)$,
${\stackrel{\sim}\sigma}_m(x)$ are replaced by  suitable
weaker notions. Precisely, for \textbf{partially greedy bases} we
will prove that a basis is partially greedy if and only if
$(\ref{g})$, $(\ref{ag})$ are satisfied where $\sigma_m(x)$ and
${\stackrel{\sim}\sigma}_m(x)$ are replaced by the weaker notions
$\sigma_m^L(x)$ and ${\stackrel{\sim}\sigma}_m^L(x)$ respectively.

Motivated by these new characterizations of partially greedy bases
we introduce the notion of \textbf{reverse partially greedy basis}.
We say a basis is reverse partially greedy if it satisfies
$(\ref{ag})$ where ${\stackrel{\sim}\sigma}_m(x)$ is replaced by
${\stackrel{\sim}\sigma}_m^R(x)$. We will also  prove
characterizations of reverse partially greedy bases which are analogues
of the characterizations of partially greedy bases mentioned above
and proved in \cite{S}.

In section 3 we will study the characterizations of $1$-partially greedy
and $1$-reverse partially greedy bases. In the last section we will study the weak and branch versions of
the new characterizations of almost greedy bases.

\section{Characterizations of partially greedy, reverse partially
greedy and almost greedy bases} \label{sec: 2}

First we recall a  few results concerning quasi-greedy bases from
\cite{S}.
 In fact these results do not require the basis to be a Schauder basis  and are valid, more generally, for
biorthogonal systems $((e_n),(e_n^*))$ such that $a < \|e_n\|,
\|e_n^*\| < b$ for some positive  constants $a,b$, $(e_n)$ has dense
linear span in $X$, and the map $x \mapsto  (e_n^*(x))$  is
injective. We say that $(e_n)$ is a \textbf{bounded Markushevich
basis}.

Set  $\N^m:=\{A \subset \N:~|A|=m\}$ and $\N^{<
\infty}:=\cup_{m=0}^{\infty}\N^m$.
\begin{lem}
Suppose that $(e_n)$ is a quasi-greedy basis with constant $K$ and
$A\in \N^{< \infty}$. Then, for every choice of signs $\varepsilon_j
=\pm 1$, we have
\begin{align}\label{uc}
\frac{1}{2K}\|\sum_{j\in A}e_j\|\leq \|\sum_{j\in A} \varepsilon_j
e_j\|\leq 2K \|\sum_{j\in A} e_j\|,
\end{align}
and hence for any scalars $(a_i)_{i\in A}$,

\begin{align}\label{cuc}
\| \sum_{j\in A} a_je_j\|\leq 2K~ \underset{j\in
A}{max}|a_j|\|\sum_{j\in A} e_j\ \|.
\end{align}
\end{lem}

\begin{lem}
Suppose that $(e_n)$ is a quasi-greedy basis with constant $K$ and
$x\in X$ has greedy ordering $\rho$. Then
\begin{align}\label{min}
|e_{\rho(m)}^*(x)| \left\|\sum_1^m e_{\rho(i)}\right\|\leq 4K^2 \| x
\|.
\end{align}
\end{lem}

Recall that the fundamental function $\phi(n)$ of a basis $(e_n)$ is
defined by
\begin{align*}
\phi(n)=\underset{|A|\leq n}{sup}\|1_A\|.
\end{align*}
 Note that $\phi$ is a subadditive function, i.e., $\phi(m+n)\leq\phi(m)+\phi(n)$, and satisfies
$\phi(m) \ge(m/n)\phi(n)$ for all positive integers $m,n$ with $m\le n$.

If $(e_n)$ is an  almost greedy basis then from \eqref{min} it follows
that there exists a constant $C$ such that
\begin{align}\label{agmin}
\phi(|A|)min|a_i|\leq C\|\sum_{i \in A} a_i e_i\|
\end{align}
for all $A\in \N^{< \infty}$ and all scalars $a_i$, $i\in A$.

In \cite{S} it was proved that a basis $(e_n)$ of a Banach space $X$
is partially greedy if and only if it is quasi-greedy and
conservative. We now give another characterization of partially
greedy bases.
\begin{thm}\label{pg}
Suppose $(e_n)$ is a basis of a Banach space $X$. Then \TFAE \bla
\item  $(e_n)$ is partially greedy.
\item  $(e_n)$ is quasi-greedy and conservative.
\item  There exists a constant $C$ such that
\begin{align}\label{qc}
\| x-G_m(x)\|\leq C {\stackrel{\sim}\sigma}_m^L(x),~x\in X,~m\in\N.
\end{align}
\el
\end{thm}
\begin{proof}
Equivalence of $(a)$ and $(b)$ was observed in \cite{S}. Clearly
$(c)$ implies that $\|G_m(x)\|\le (1+C)\|x\|$ for all $m \ge 1$ and for all $x \in X$, and hence that  the basis is quasi-greedy. We now prove that
$(c)$ implies that the basis is conservative.

Let $A,B\in \N^{<\infty}$ with $|A|\leq |B|=m$, $A<B$ and $m\in \N$.
Consider $x={\sum}_{i\in A}e_i+(1+\varepsilon){\sum}_{i\in B}e_i$
for any $\varepsilon>0$. Then $G_m(x)=(1+\varepsilon){\sum}_{i\in
B}e_i$. Now from  $(c)$, we have $\| 1_A \|\leq C
{\stackrel{\sim}\sigma}_m^L(x)\leq C(1+\varepsilon)\| 1_B \|$. Thus,
by letting $\varepsilon\rightarrow 0$, we get $(e_n)$ is
conservative with constant $C$, and this gives $(c)$ implies $(a)$.

Now we will prove that $(b)$ implies $(c)$.
Let the  basis $(e_n)$  be quasi-greedy with constant $K$ and conservative
with constant $\Gamma_c$. Let $x\in X$, $m\in \N$ and
$G_m(x)=\underset{i\in \Lambda_m(x)}{\sum}e_i^*(x)e_i$.

Now choose any $A\subset \N$ with $|A|\leq m$ and $A<\al_m(x)$. Then
\begin{align*}
x-G_m(x)=x-\underset{i\in A}{\sum}e_i^*(x)e_i-\underset{i\in
\Lambda_m(x)}\sum e_i^*(x)e_i+\underset{i\in A}{\sum}e_i^*(x)e_i.
\end{align*}
We can write
\begin{align*}
\underset{i\in
\Lambda_m(x)}{\sum}e_i^*(x)e_i=G_m\left(x-\underset{i\in
A}{\sum}e_i^*(x)e_i\right),
\end{align*} and hence
$$\|
\underset{i\in
\Lambda_m(x)}{\sum}e_i^*(x)e_i\|\le  K\|x-\underset{i\in
A}{\sum}e_i^*(x)e_i\|.$$
From $(\ref{cuc})$, $(\ref{min})$ we have
\begin{align*}
\|\underset{i\in A} \sum e_i^*(x)e_i  \| &\leq 2K~\underset{i\in
A}{max}|e_i^*(x)|\| 1_A\|\\ &\leq 2K~\underset{i\in
\Lambda_m(x)}{min}|e_i^*(x)|\|1_A\|\\&\leq
8K^3\Gamma_c\|\underset{i\in \Lambda_m(x)}{\sum}e_i^*(x)e_i\|\\
& \le 8K^4\Gamma_c \|
x-\underset{i\in A}{\sum}e_i^*(x)e_i \|
      \qedhere
\end{align*}
Thus $\parallel x-G_m(x)\parallel\leq (1+K+8K^4\Gamma_c)\|
x-\underset{i\in A}{\sum}e_i^*(x)e_i \|$. This completes the proof.

\end{proof}

\bdfn A basis $(e_n)$ of a Banach space $X$ is said to be \textbf
{reverse conservative} if there exists some constant $\Gamma_r$ such
that
\begin{align*}
\| 1_A \| \leq \Gamma_r \|1_B\| \hspace{7mm} if \hspace{3mm} B<A~
and~ |A|\leq |B|.
\end{align*}
\edfn

\bdfn We say that a basis $(e_n)$ of a Banach space $X$ is \textbf
{reverse partially greedy} if there exists some constant $C$ such
that
\begin{align}\label{qr}
\parallel x-G_m(x)\parallel\leq C {\stackrel{\sim}\sigma}_m^R(x)
\end{align}
for all $x\in X$ and $m\in \N$. \edfn
\begin{ex}
We present an example of a reverse conservative basis which is not
democratic.
For each $n\in \mathbb{N}$, we define $F_n:=\{A\subset \N:|A|\leq n!,
~n!\leq A\}$ and $F:=\bigcup_{n\geq 1}F_n$.

Observe that the set $F$ is closed under spreading to the right: in fact, if
 $A,B\in\mathbb{ N}^m$ , $A\in F$ and $\min A\leq \min B$, then $B\in F$.

We define a norm on $c_{00}$ as follows
\begin{align*}
\|x\|=\max \{(|x|,1_A):A\in F\}
\end{align*}
for $x\in c_{00}$.
Let $X$ be the completion of $c_{00}$ in this norm. The canonical
basis $(e_n)$ of $X$  is  normalized and $1$-unconditional.

From the right spreading property of $F$, it follows that $\|x\|\leq
\|y\|$, where $y=\sum a_i e_{n_i}$ is a spread of $x=\sum a_ie_i$
with $n_1<n_2<\ldots$. In particular, if $A<B$ and $|A|\leq |B|$
then $\|1_A\|\leq \|1_B\|$. Thus $(e_n)$ is a $1$-conservative basis.
However, $\|1_{[1,n!]}\|=(n-1)!$, while $\|1_{[n!+1,2n!]}\|=n!$. So
$(e_n)$ is not a democratic basis.

Now consider the dual norm $\|\cdot\|^*$. Since $\|x\|\leq \|y\|$ it
follows easily that $\|x\|^*\geq\|y\|^*$. In particular,
$\|\cdot\|^*$ is reverse conservative. But $\|1_{[1,n!]}\|=(n-1)!$
implies $\|1_{[1,n!]}\|^*\geq n$ while $\|1_{[n!+1,2n!]}\|^*= 1$
since $[n!+1,2n!]\in F_n$. So the dual norm is not democratic.\\
\end{ex}
The proof of the following characterization of  reverse partially greedy bases
 is similar to the proof  of Theorem~\ref{pg}.
\begin{thm}
Let $(e_n)$ be a basis of Banach space $X$. Then \TFAE \bla
\item  $(e_n)$ is reverse partially greedy.
\item  $(e_n)$ is quasi-greedy and reverse conservative.

\el
\end{thm}

We now prove that if a basis $(e_n)$ of a Banach space $X$ is both
conservative and reverse conservative  then it is democratic.
If $X$ is an infinite-dimensional Banach space then for given $A,B\in
\N^m$ we can find $C \in \N^m$ with $A<C$ and $B<C$. Now by the
conservative and reverse conservative properties of the basis we can
easily conclude that the basis is democratic. We now give another
proof of this fact which has the advantage of  working  for a \textit{finite} basis as well.

\begin{lem}\label{crd}
Let $(e_n)$ be a basis of Banach space $X$. If $(e_n)$ is both
conservative and reverse conservative then $(e_n)$ is democratic.
\end{lem}
\begin{proof}
Let there exist a constant $\Gamma$ such that for any two sets
$A,B\in \N^{< \infty}$, with $|A|\leq |B|$, we have
\begin{align*}
\| 1_A \| \leq \Gamma \|1_B\| \hspace{7mm} if \hspace{3mm} A<B
\hspace{3mm} or \hspace{3mm}A> B.
\end{align*}

Choose any two sets  $A,B\in \N^m$.  Let $A=\{a_1<a_2<\ldots <a_m\}$
and $B=\{b_1<b_2<\ldots <b_m\}$.
If $A <  B$ or $B< A$, then
$\|1_A\| \leq \Gamma \|1_B\|$ and $\|1_B\| \leq \Gamma\|1_A\|$.
So, for the rest of the proof we assume that
this is not the case.

Hence  $a_1\leq b_m$. If $a_1=b_m$ then we can write $A=A_1\cup
A_2$, $B=B_1\cup B_2$ where $A_1=\{a_1\}$, $A_2=\{a_2,\ldots,a_m\}$,
$B_1=\{b_m\}$, $B_2=\{b_1,\ldots,b_{m-1}\}$ and $B_2<A_2$.

Let $K_b$ be the basis constant for $(e_n)$. Then
\begin{align*}
\|1_A\|\leq \|1_{A_1}\| + \|1_{A_2}\|\leq
\|1_{B_1}\|+\Gamma\|1_{B_2}\|\leq (K_b(1+\Gamma)+1) \|1_B\|
\end{align*}
and
\begin{align*}
\|1_B\|\leq \|1_{B_1}\| + \|1_{B_2}\|\leq
\|1_{A_1}\|+\Gamma\|1_{A_2}\|\leq (K_b(1+\Gamma)+\Gamma) \|1_A\|.
\end{align*}

Now consider the case $a_1<b_m$. If we compare $a_2$ with $b_{m-1}$
then there can be three possibilities: $a_2<b_{m-1}$, $a_2>b_{m-1}$
or $a_2=b_{m-1}$. If $a_2>b_{m-1}$ or $a_2=b_{m-1}$ then we will
stop the process; otherwise we will continue in the same manner. By
the assumptions on the sets $A,B$ we can find the first $j$, $1\leq
j <  m$, such that either $a_{j+1}>b_{m-j}$ or $a_{j+1}=b_{m-j}$.

Thus we can write \textit{ either}

$A=A_1\cup A_2$ and $B=B_1\cup B_2$ where $|A_i|=|B_i|$, $A_1<B_1$,
$A_2>B_2$,  $A_1=\{a_1<\ldots<a_j\}<B_1=\{b_{m-j+1}<\ldots<b_m\}$
and $A_2=\{a_{j+1}<\ldots<a_m\}>B_2=\{b_{1}<\ldots<b_{m-j}\}$

\textit{or}

 $A=A_1\cup A_2\cup A_3$, $B=B_1\cup  B_2 \cup B_3$  where
$A_1=\{a_1<\ldots<a_j\}<B_1=\{b_{m-j+1}<\ldots<b_m\}$,
$A_2=\{a_{j+1}\}=\{b_{m-j}\}=B_2$ and
$A_3=\{a_{j+2}<\ldots<a_m\}>B_2=\{b_{1}<\ldots<b_{m-j-1}\}.$

For the first case we have
\begin{align*}
\| 1_{A_i} \| \leq \Gamma \|1_{B_i}\| ~\text{and}~ \| 1_{B_i} \| \leq
\Gamma \|1_{A_i}\| \hspace{7mm} for \hspace{3mm} i=1,2.
\end{align*}

Now we can write
\begin{align*}
\|1_A\|\leq \|1_{A_1}\| + \|1_{A_2}\|\leq
\Gamma(\|1_{B_1}\|+\|1_{B_2}\|)\leq \Gamma (2K_b+1) \|1_B\|
\end{align*}
and
\begin{align*}
\|1_B\|\leq \|1_{B_1}\| + \|1_{B_2}\|\leq
\Gamma(\|1_{A_1}\|+\|1_{A_2}\|)\leq \Gamma (2K_b+1) \|1_A\|.
\end{align*}

 For the second case we get
\begin{align*}
\|1_A\|&\leq \|1_{A_1}\| + \|1_{A_2}\| +\|1_{A_3}\|\\&\leq
\Gamma(\|1_{B_1}\|+\|1_{B_3}\|)+\|1_{B_2}\|\\ &\leq (\Gamma
(2K_b+1)+2K_b )\|1_B\|
\end{align*}
and
 \begin{align*}
\|1_B\|&\leq \|1_{B_1}\| + \|1_{B_2}\| +\|1_{B_3}\|\\&\leq
\Gamma(\|1_{A_1}\|+\|1_{A_3}\|)+\|1_{A_2}\|\\&\leq (\Gamma
(2K_b+1)+2K_b) \|1_A\|.
\end{align*}
\end{proof}

Next, we consider  basis conditions that are formally  stronger than  (\ref{qc}) and
(\ref{qr}). We will say that a basis $(e_n)$ satisfies

\bla
\item {\textbf{\em property $(*)$}} if there exists a constant $C$ such that
\begin{align*}
\|x-G_m(x)\| \leq ~C  \sigma_m^L(x),~x\in X, ~m\in \N
\end{align*}

\item {\textbf{\em  property $(**)$}} if there exists a constant $C$ such that
\begin{align*}
\|x-G_m(x)\| \leq ~C \sigma_m^R(x),~x\in X, ~m\in \N.
\end{align*}
\el

While properties $(*)$ and $(**)$ appear to be stronger than $(\ref{qc})$
and $(\ref{qr})$ respectively, the following results prove that this is not the case.

\begin{thm}\label{gpg}
Let $(e_n)$ be a basis of a Banach space $X$. Then \TFAE \bla
\item $(e_n)$ is partially greedy.
\item $(e_n)$ satisfies property $(*)$.
\el
\end{thm}

\begin{proof}
Theorem~\ref{pg} shows that $(b)$ implies $(a)$.

We now prove that $(a)$ implies $(b)$. Let $K_b,K,\Gamma_c$ be the
basis constant, quasi-greedy constant and conservative constant
respectively. Let $x\in X$ and $A\subseteq \{1,\ldots,\al_m(x)-1\}$ with $|A|\leq
m$. Consider $y=\sum e_i^*(y)e_i$ where $e_i^*(y)=e_i^*(x)$ for all
$i\not \in A$. To prove the theorem it is sufficient to show that
\begin{equation} \label{eq: y_normestimate}
 \|x - G_m(x)\| \le C\|y\| ,\end{equation}
where $C$ depends only on $K,K_b$ and $\Gamma_c$.
We can write $A=A_1 \cup A_2$, where
\begin{align*}
A_1=\{i\in A:|e_i^*(y)|>|e_{\rho(m)}^*(x)|\}
\end{align*}
and $A_2=A\setminus A_1$.
Then $$\|\sum_{i\in A_1}e_i^*(y)e_i\|=\|G_{|A_1|}(\sum_{i=1}^{\alpha_m(x)-1}e^*_i(y)e_i)\|\leq KK_b\|y\|.$$
Also from $(\ref{cuc})$ and $(\ref{min})$ we have
\begin{align*}
\|\underset{i\in A_2}{\sum}e_i^*(y) e_i\| & \leq 2K~\underset{i\in A_2}{max}|e_i^*(y)| \|1_{A_2}\|\\
&\leq 2K~\underset{i\in \Lambda_m(x)}{min}|e_i^*(x)|\|1_{A_2}\| \\
&\leq 8K^3\Gamma_c\| \underset{i\in \Lambda_m(x)}{\sum}e_i^*(x)e_i\|\\
&= 8K^3\Gamma_c\| G_m(\sum_{\al_m(x)}^{\iy}e_i^*(x)e_i)\|\\
&\leq 8K^4\Gamma_c\|\sum_{\al_m(x)}^{\iy}e_i^*(x)e_i\|\\
&=8K^4\Gamma_c\|\sum_{\al_m(x)}^{\iy}e_i^*(y)e_i\|\\
&\leq 8K^4\Gamma_c(K_b+1) \|y\|.
\end{align*}
Thus
\begin{align*}
\|\underset{i\in A}{\sum}e_i^*(y) e_i\| \leq
(KK_b+8K^4\Gamma_c(K_b+1))\|y\|.
\end{align*}

Now we will find estimates of  $\|G_m(x)\|$, $\|\underset{i\in
A}{\sum}e_i^*(x) e_i\|$ and $\|x\|$ in terms of $\|y\|$.
\begin{align*}
\|G_m(x)\|&=\|G_m(\sum_{\al_m(x)}^{\iy}e_i^*(x)e_i)\|\\&  \leq K
\|\sum_{\al_m(x)}^{\iy}e_i^*(x)e_i \|\\&=K
\|\sum_{\al_m(x)}^{\iy}e_i^*(y)e_i \|\\& \leq K(K_b+1)\|y\|.
\end{align*}

From $(\ref{cuc})$ and $(\ref{min})$ we can write
\begin{align*}
\|\underset{i\in A}{\sum}e_i^*(x)e_i\|  &\leq 2K~\underset{i\in A}{max}|e_i^*(x)| \|1_{A}\|\\
&\leq 2K~\underset{i\in \Lambda_m(x)}{min}|e_i^*(x)|\|1_{A}\| \\
&\leq 8K^3\Gamma_c\| \underset{i\in \Lambda_m(x)}{\sum}e_i^*(x)e_i\|\\
&=8K^3\Gamma_c \|G_m(\sum_{\al_m(x)}^{\iy}e_i^*(x)e_i)\|\\
&\leq 8K^4\Gamma_c\|\sum_{\al_m(x)}^{\iy}e_i^*(x)e_i\|\\
&= 8K^4\Gamma_c\|\sum_{\al_m(x)}^{\iy}e_i^*(y)e_i\|\\
&\leq 8K^4\Gamma_c(K_b+1) \|y\|.
\end{align*}
Thus
\begin{align*}
 \|x\|= &\|y- \underset{i\in A}{\sum}e_i^*(y)e_i+\underset{i\in A}{\sum}e_i^*(x)e_i\|\\
                               &\leq \|y\|+\|\underset{i\in A}{\sum}e_i^*(y)e_i\|+\|\underset{i\in A}{\sum}e_i^*(x)e_i\|\\
                               &\leq(KK_b+16K^4\Gamma_c(K_b+1)+1) \|y\|.
\end{align*}

From these estimates we have
\begin{align*}
\|x-G_m(x)\|& \leq\| x\|+\|G_m(x) \|\\
&\leq (KK_b+16K^4\Gamma_c(K_b+1)+K(K_b+1)+1)\|y\|.
\end{align*}
\end{proof}

\Rem The proof  of Theorem~\ref{gpg} shows that for a partially greedy
basis $(e_n)$ we have $\|x\|\leq C_1\|y\|$ for some constant $C_1$,
which is stronger than \eqref{eq: y_normestimate}.
\end{rem}

Similar proofs give the following characterizations of  reverse
partially greedy bases and almost greedy bases.
\begin{thm}
Let $(e_n)$ be a basis of a Banach space $X$. Then \TFAE \bla
\item $(e_n)$ is quasi-greedy and  reverse conservative.
\item $(e_n)$ satisfies property $(**)$.
\el
\end{thm}

\begin{thm}\label{gag}
Let $(e_n)$ be a basis of a Banach space $X$. Then \TFAE \bla
\item $(e_n)$ is almost greedy.
\item There exists a constant $C$ such that for any $x\in X$,
$A\subset\{1,\ldots, \al_m(x)-1\}\cup \{\be_m(x)+1,\ldots\}, ~|A|\leq m$
and $a_i\in \mathbb{R}$, $i\in A$, we have
\begin{align*}
\|x-G_m(x)\| \leq ~C   \|x-\underset{i\in A}{\sum}a_ie_i\|.
\end{align*} \end{list}
\end{thm}

Note that  $(e_n)$ is required to be a Schauder basis in the proofs
of the above results  as the basis constant $K_b$ appears in certain
estimates. However, by replacing $(b)$ in Theorem~\ref{gag} by a
stronger condition we can get a characterization of almost greedy
bounded Markushevic bases (see the first paragraph of
Section~\ref{sec: 2}).
\begin{thm} \label{thm: agcharacterization}
Let $(e_n)$ be a bounded  Markushevich basis for a  Banach space $X$. Then \TFAE \bla
\item $(e_n)$ is almost greedy.
\item Let $0 \le \lambda<1$. Then there exists a constant $C$ such that for any $x\in X$,
$A\subset \mathbb{N}$ with $|A|\leq m$ and  $|A\cap \Lambda_m(x)|\leq \lambda m$,
and any  $a_i\in \mathbb{R}$, $i\in A$, we have
\begin{align*}
\|x-G_m(x)\| \leq ~C   \|x-\underset{i\in A}{\sum}a_ie_i\|.
\end{align*}
\el
\end{thm}

\begin{proof}
First we show that $(b)$ implies $(a)$.  Setting $\lambda=0$, $(b)$ implies $\|G_m(x)\| \le (1+C)\|x\|$ for all $m\ge1$ and for all $x \in X$, and hence that
$(e_n)$ is quasi-greedy. Now we show, setting $\lambda=0$ again,  that $(b)$ implies that $(e_n)$ is democratic. Let $A,B \in \mathbb{N}^m$. First suppose that $A$ and $B$ are disjoint.
Then, applying $(b)$ to $x = 1_B+ (1+\varepsilon)1_A$ yields. $\|1_B\| \leq C\|1_A\|$. For the general case, note that $(b)$ implies $\|1_{E}\| \le (1+C)\|1_A\|$ for all $E\subseteq A$. Hence
$$\| 1_B \| \le \|1_{A \cap B}\| + \|1_{B\setminus A}\| \leq (1+C)\|1_A\| + C\|1_{A\setminus B}\| \leq (1+C)^2\|1_A\|.$$
Hence $(e_n)$ is democratic. So $(e_n)$ is quasi-greedy and democratic, and hence almost greedy.

Now we show that $(a)$ implies $(b)$.
Let $\phi(n)$ be the fundamental function for an almost greedy basis
$(e_n)$.  In the following inequalities the constants $C_1,C_2$ etc. depend only on $\lambda$ and  the quasi-greedy and democratic constants of $(e_n)$.
 Let $A \subset
\mathbb{N}$ with $|A\cap \Lambda_m(x)|\leq \lambda m$ and $|A|\leq
m$. We may assume without loss of generality that $\lambda m \in \mathbb{N}$.
 Then for all coefficients $(a_i)_{i\in A}$,
\begin{align*}
\| x-\underset{i\in A}{\sum}a_ie_i \|& \geq C_1 \|G_{(1-\lambda)m}(x-\underset{i\in A}{\sum}a_ie_i)\|\\
                               &\geq C_2 |e_{\rho(m)}^*(x)|\phi((1-\lambda)m)\\
                               &\geq  C_2(1-\lambda)|e_{\rho(m)}^*(x)|\phi(m)\\
                               &\geq C_3 \|G_{2m}(x)-G_m(x)\|\\
                               &= C_3\|(x-G_{m}(x))-(x-G_{2m}(x))\|.
\end{align*}
The first inequality follows from quasi-greediness of $(e_n)$.
The second inequality follows from  $(\ref{agmin})$ and the fact
that $|A\cap \Lambda_m(x)|\leq \lambda m$, so the largest
$(1-\lambda)m$ coefficients of $x-\underset{i\in A}{\sum}a_ie_i$  are at least $|e_{\rho(m)}^*(x)|$.
The third inequality follows from the fact that $\phi(r) \ge \dfrac{r}{s} \phi(s)$ for positive integers $r \le s$.
The
fourth inequality follows from $(\ref{cuc})$.

Now consider two cases. First suppose that $\|x-G_{2m}(x)\|\leq
\frac{1}{2} \|x-G_{m}(x)\|$. Then by the  triangle inequality
\begin{align*}
\| x-\underset{i\in A}{\sum}a_ie_i \| \geq
C_3(\|(x-G_{m}(x)\|-\|x-G_{2m}(x)\|)\geq \frac{C_3}{2}\|x-G_m(x)\|.
\end{align*}
Now suppose  $\|x-G_{2m}(x)\|>\frac{1}{2} \|x-G_{m}(x)\|$. Then by
 \cite[Theorem~3.3]{S}, we have
\begin{align*}
\|x-G_m(x)\|\leq 2\|x-G_{2m}(x)\|\leq C_4\sigma_m(x)\leq C_4 \|
x-\underset{i\in A}{\sum}a_ie_i \|.
\end{align*}

Thus in both the cases we get
\begin{align*}
\|x-G_m(x)\|\leq C \| x-\underset{i\in A}{\sum}a_ie_i \|
\end{align*}
for some constant $C$.
\end{proof}
\begin{rem} Theorem~\ref{thm: agcharacterization} complements \cite[Theorem~3.2]{DKK}  which states that if $(e_n)$ is an  almost greedy basis then there exists $C>0$ such that  for all $x \in X$ and $n \ge 1$ there exist scalars $a_i$ ($i \in \Lambda_n(x)$) such that
$$\|x - \sum_{i \in \Lambda_n(x)}a_i e_i\| \le C \sigma_n(x).$$ \end{rem}
The following characterization of greedy bases was proved in
\cite{BB}.
\begin{thm}
Let $(e_n)$ be a basis of a Banach space $X$. Then $(e_n)$ is a greedy
basis if and only if there exists some constant $C$ such that
\begin{align*}
\|x-G_m(x)\| \leq C   \{d(x,a1_A):A\subset \N,~a\in \R,~|A|=m\}
\end{align*}
for all $x\in X$ and $m\in \N$.
\end{thm}

We now prove the similar result for almost greedy bases.
\begin{thm}
Let $(e_n)$ be a basis of a Banach space $X$. Then \TFAE \bla
\item $(e_n)$ is almost greedy basis.
\item There exists some constant $C$ such that for all
$x\in X$, $m\in \N$, $a\in \R$, 
 $A \subset\mathbb{N}$ with
$|A|\leq m$, and either  $A <\alpha_m(x)$ or $A > \beta_m(x)$,
 we have
\begin{align}\label{thm: greedy}
\|x-G_m(x)\| \leq C d(x,a1_A).
\end{align}
\el
\end{thm}

\begin{proof}
Theorem~\ref{gag} shows that $(a)$ implies $(b)$.

Now we prove that $(b)$ implies $(a)$. From (\ref{thm: greedy}) it
follows that $\|G_m(x)\| \leq (C+1) \|x\|$ for all $x\in X$ and $m\in
\N$. Hence $(e_n)$ is quasi-greedy.  For any $\varepsilon >0$, $A,B\in \N^{< \infty}$ with $A<B$ and
$|A| \leq |B|$ consider $x:=1_A+(1+\varepsilon)1_B$. Clearly from
(\ref{thm: greedy}) it follows that $\|1_A\|=\|x-G_{|B|}(x)\| \leq
C(1+\varepsilon) \|1_B\|$. Since $\varepsilon>0$ is arbitrary, we get
$\|1_A\|\leq C\|1_B\|$.

Similarly, if we consider any $A,B\in \N^{<\infty}$ with $A>B$ and
$|A| \leq |B|$ then we can prove that $\|1_A\|\leq C\|1_B\|$.

Thus from Lemma~\ref{crd} it follows that the basis is democratic
and this proves that $(b)$ implies $(a)$.
\end{proof}

\section{Characterization of 1-partially and 1-reverse partially
greedy bases}

\begin{lem}\label{l2}
A basis $(e_n)$ of Banach space $X$ satisfies (\ref{qc})
 with $C=1$ if and only if for any $x\in X$ and
$j<\al_1(x)$
\begin{equation} \label{eq: constantone}
\|x-G_1(x)\|\leq \min(\|x\|, \|x-e_j^*(x)e_j\|).
\end{equation}
\end{lem}
\begin{proof}
If (\ref{qc})
is satisfied with $C=1$ then \eqref{eq: constantone} follows immediately.

Conversely, suppose that \eqref{eq: constantone} holds.  Note that  iterating \eqref{eq: constantone}
yields  $$\|x - G_m(x)\| \le \|x -G_k(x)\| \le  \|x\|$$ for all $m \ge k \ge 1$.
Suppose that $A < \alpha_m(x)$ and $0 \le |A| =  k \leq m$.  Let $A=\{j_1,\dots,j_k\}$, where
$j_i < \alpha_m(x)$  for $1 \le i \le k$. Let $y=x-\sum_{j=1}^{k-1}e_{\rho(j)}^*(x)e_{\rho(j)}$.
 Then
\begin{align*}
\|x-G_k(x)\|&=\|(x-\sum_{i=1}^{k-1}e_{\rho(i)}^*(x)e_{\rho(i)})-e_{\rho(k)}^*(x)e_{\rho(k)}\|\\ &=\|y-G_1(y)\|\\&\leq
\|y-e_{j_1}^*(y)e_{j_1}\|\\&=\|x-\sum_{i=1}^{k-1}e_{\rho(i)}^*(x)e_{\rho(i)}-e_{j_1}^*(x)e_{j_1}\|
\end{align*}
since  $j_1<\alpha_m(x) \leq \rho(m)$. Continuing in this way we get
\begin{align*}
\|x-G_k(x)\|\leq \|x-\sum_{i=1}^ke_{j_i}^*(x)e_{j_i}\|,
\end{align*}
and hence $$\|x-G_m(x)\|\leq \|x - G_k(x)\| \leq  \|x-\sum_{i=1}^ke_{j_i}^*(x)e_{j_i}\|.$$
Thus, $\|x-G_m(x)\|\leq {\stackrel{\sim}\sigma}_m^L(x)$.
\end{proof}

\begin{thm}
A basis $(e_n)$ of a Banach space is $1$-partially greedy  (i.e., satisfies $(\ref{qc})$ with $C=1$)
if and only if for any $x\in X$ and  $j,k\in \N\setminus supp(x)$,
with $j < k$,
\begin{equation} \label{eq: propertyAcharacterization}
\max(\|x\|,\|x+se_j\|) \leq\|x+te_k\|
\end{equation}
where $|s|= |t| \ge \max|e^*_i(x)|$.
\end{thm}
\begin{proof}
Suppose $(e_n)$  satisfies $(\ref{qc})$ with $C=1$. Let us verify
\eqref{eq: propertyAcharacterization}. For any $\varepsilon >0$
consider $y=x+se_j+(1+\varepsilon)te_k$ where $j,k\in \N\setminus
supp(x)$, $j<k$ and $|s|= |t| \ge \max|e^*_i(x)|$. Clearly,
$G_1(y)=(1+\varepsilon)te_k$ and thus
\begin{align*}
\|x+se_j\|=\|y-G_1(y)\|\leq \|y-e_j^*(y)e_j\|=\|x+(1+\varepsilon)te_k\|.
\end{align*}
Letting $\varepsilon \longrightarrow 0$, we get $\|x+se_j\|\leq \|x+te_k\|$. Similarly, setting $z = x + (1+\varepsilon)te_k$,
$$\|x\| = \|z - G_1(z)\| \le \|z\| = \|x + (1+\varepsilon)te_k\|,$$
and hence $\|x\| \le \|x+ te_k\|$. Thus, \eqref{eq: propertyAcharacterization} is satisfied.

Conversely, suppose that  \eqref{eq: propertyAcharacterization} is satisfied. Let $x \in X$. Suppose that $G_1(x) = e_k^*(x)e_k$
and suppose $j < k$. Consider $y = x - G_1(x)$. Then by \eqref{eq: propertyAcharacterization},
$$ \|x - G_1(x)\| =\|y\| \le \|y + e_k^*(x)e_k\| = \|x\|.$$
Now consider $ z = x - e_j^*(x)e_j - e^*_k(x)e_k$. Then by  \eqref{eq: propertyAcharacterization}
$$\|z  \pm se_j\| \le \|z + se_k\|,$$
where $s = e_k^*(x)$. By convexity, since $|e^*_j(x)| \le |s|$,
$$\|z + e_j^*(x)e_j\| \le \|z +se_k\|,$$
i.e.,
$$\|x - G_1(x)\| \le \|x - e_j^*(x)e_j\|.$$ Hence $(e_n)$ satisfies \eqref{eq: constantone}. So, by Lemma~\ref{l2},  $(e_n)$ satisfies
\eqref{qc} with $C=1$.
  \end{proof}

Similar arguments yields the following results for $1$-reverse
partially greedy bases.

\begin{lem}
A basis $(e_n)$ of Banach space $X$ satisfies (\ref{qr})
 with $C=1$ if and only if for any $x\in X$ and
$j>\be_1(x)$
\begin{equation} \label{eq: reverseconstantone}
\|x-G_1(x)\|\leq \min(\|x\|, \|x-e_j^*(x)e_j\|).
\end{equation}
\end{lem}
\begin{thm}
A basis $(e_n)$ of a Banach space is $1$-reverse partially greedy  (i.e.,  satisfies $(\ref{qr})$ with $C=1$)  if and only if for any $x\in X$ and  $j,k\in \N\setminus supp(x)$,
with $j >k$,
\begin{equation} \label{eq: reversepropertyAcharacterization}
\max(\|x\|,\|x+se_j\|) \leq\|x+te_k\|
\end{equation}
where $|s|= |t| \ge \max|e^*_i(x)|$.
\end{thm}

\section{Branch almost greedy and weak almost
greedy bases}

In this section we will consider the \textbf{WTGA} and \textbf{BGA}. The  BGA
is a more restrictive form of WTGA. First we  recall the definition of the  BGA from
\cite{DKSW}. Let $X$ be a finite dimensional or separable infinite-dimensional
Banach space. Let $(e_n)$ be a bounded Markushevich
basis. For a given
weakness parameter $\tau$, $0<\tau<1$ and $0\not=x\in X$, define
\begin{align*}
\mathcal{A}^{\tau}(x):=\{n\in \N:|e_n^*(x)|\geq \tau \underset{n\geq
1}{max}|e_n^*(x)|\}.
\end{align*}
Let $\mathcal{G}^{\tau}:X\setminus \{0\}\longrightarrow \N$ be any
mapping satisfying the following conditions: \bla
\item $\mathcal{G}^{\tau}(x) \in \mathcal{A}^{\tau}(x)$
\item $\mathcal{G}^{\tau}(\lambda x)=\mathcal{G}^{\tau}(x)$ for all
$\lambda\ne0$
\item if $\mathcal{A}^{\tau}(y)=\mathcal{A}^{\tau}(x)$ and
$e_i^*(y)=e_i^*(x)$ for all $i \in \mathcal{A}^{\tau}(x)$ then
$\mathcal{G}^{\tau}(y)=\mathcal{G}^{\tau}(x)$. \el

In this algorithm $\mathcal{G}^{\tau}(x)$ is the index of first
selected coefficient. The subsequent  coefficients are selected by iterating
the algorithm on the residuals.

In \cite{DKSW} the authors defined the following notions of
\textbf{branch quasi-greedy $(BQG(\tau))$} and \textbf{branch almost
greedy $(BAG(\tau))$} for a given weakness parameter $\tau$
 \bla
 \item $(e_n)$ is said to be $BQG(\tau)$ with constant $C$ such that
\begin{align*}
\|\mathcal{G}_m^{\tau}(x)\|\leq C\|x\|,~x\in X,~m\in\N.
\end{align*}
\item $(e_n)$ is said to be $BAG(\tau)$ with constant $C$ such that
\begin{align*}
\|x-\mathcal{G}_m^{\tau}(x)\|\leq
C{\stackrel{\sim}\sigma}_m(x),~x\in X,~m\in\N.
\end{align*}
\el They proved that if any algebraic basis $(e_i)_{i=1}^N$ of a
$N$-dimensional normed space is $BAG(\tau)$ then the basis is both
quasi-greedy and democratic. Now we will prove the similar results
for the weaker notions than $BAG(\tau)$.

For Theorem~\ref{t1} we will consider the $BGA$ and for Theorem
\ref{t3} we  will work with the  $WTGA$. If
$\mathcal{G}_m^\tau(x)=\sum_{n\in B_m^\tau(x)} e_n^*(x)e_n$, then we
denote $\al_m^{\tau}(x)=min~ B_m^\tau(x)$ and $\be_m^{\tau}(x)=max
~B_m^\tau(x)$.

To prove Theorem~\ref{t1} first we recall a few results  from \cite{DKSW}. \bdfn Let
$0<\tau<1$. Then a basis $(e_i)_{i=1}^N$ is said to have property
$P(\tau)$ if there exists some constant $C$ such that for all sets
$A\subset \{1,\ldots,N\}$ and for all scalars $(a_i)_{i\in A}$ with
$1\leq |a_i|\leq \frac{1}{\tau^2}$, we have
\begin{align}\label{propertyp}
\underset{\pm}{max}\|\sum_{i\in A}\pm e_i\| \leq C \|\sum_{i\in A}a_ie_i\|.
\end{align}
\edfn
\begin{Prop}\label{propp}
Let $(e_i)_{i=1}^N$  be  $BQG(\tau)$ and have property $P(\tau)$. Then $(e_i)_{i=1}^N$ is quasi-greedy.
\end{Prop}
\begin{lem}\label{lemp}
Let $0<\tau<1$. Suppose that $(e_i)_{i=1}^N$ is $BQG(\tau)$ and democratic. If (\ref{propertyp}) is satisfied for all $A$ with
$|A|\leq N/2$ and all scalars $(a_i)_{i\in A}$ with $1\leq |a_i|\leq \frac{1}{\tau^2}$ then $(e_i)_{i=1}^N$ has property $P(\tau)$.
\end{lem}
\begin{thm}\label{t1}
Let $0<\tau <1$. Suppose that $(e_i)_{i=1}^{N}$ is a basis of $X$
and there exists a constant $C$ such that
\begin{align}\label{bage1}
\|x-\mathcal{G}_m^\tau(x)\|\leq C \|x-P_A(x)\|
\end{align}
where $x\in X$, $0\leq k \leq N$, $A \subset\mathbb{N}$ with  $|A|\leq
m$, and either $A< \alpha^\tau_m(x)$ or $A>\beta_m^\tau(x)$. Then $(e_i)_{i=1}^{N}$ is almost greedy.
\end{thm}

\begin{proof}
Clearly, it follows from (\ref{bage1}) that $(e_i)_{i=1}^N$ is
$BQG(\tau)$ with constant $1+C$.  First we prove that (\ref{propertyp}) is
satisfied for  all $A\subset\{1,\cdots,N\}$ with
$|A|\leq N/2$ and all scalars $(a_i)_{i\in A}$ with $1\leq |a_i|\leq \frac{1}{\tau^2}$. Choose any $A\subset\{1,\cdots,N\}$ with $|A|=n\leq N/2$ and
scalars $(a_i)_{i\in A}$ such that $1\leq |a_i| \leq1/{\tau^2}$. Now choose
$D\subset\{1,\ldots, N\}$ disjoint from $A$ and $|A|=|D|$. Let
$A=\{a_1<\ldots<a_k\}$ and $D=\{d_1<\ldots< d_k\}$ and without loss
of generality we can assume that $a_1<d_k$. Now by using the fact
that $A\cap D=\emptyset$ and arguments of Lemma~\ref{crd} we can
find $j$, $0\leq j\leq k$ such that
\begin{align*}
A=A_1\cup A_2,~D=D_1\cup D_2
\end{align*}
where $|A_i|=|D_i|=k_i$ for $i=1,2$, $A_1=\{a_1,\ldots,a_j\}$,
$A_2=\{a_{j+1},\ldots, a_k\}$, $D_1=\{d_{k-j+1},\dots,d_k\}$,
$D_2=\{d_1,\ldots,d_{k-j}\}$, $A_1<D_1$ and $A_2>D_2$.

Consider $x_j=\theta \sum_{i\in D_j}e_i+\sum_{i\in A_j}a_ie_i$ where
$j=1,2$ and $0<\theta<\tau$. Then
$\mathcal{G}_{k_j}^\tau(x_j)=\sum_{i\in A_j}a_ie_i$ and thus for
$j=1,2$  \eqref{bage1} yields
\begin{align}\label{es1}
\theta\|\sum_{i\in D_j}e_i\|\leq C\|\sum_{i\in A_j}a_ie_i\|.
\end{align}
If we consider $y_j=\sum_{i\in D_j}e_i+\theta \sum_{i\in A_j}\pm
e_i$, then $\mathcal{G}_{k_j}^\tau(y_j)=\sum_{i\in D_j}e_i$ and thus
for $j=1,2$ \eqref{bage1} yields
\begin{align}\label{es2}
\theta\|\sum_{i\in A_j}\pm e_i\|\leq C\|\sum_{i\in D_j}e_i\|.
\end{align}

From $(\ref{es1})$ and $(\ref{es2})$ we have
\begin{align*}
\theta\|\sum_{i\in A}\pm e_i\| &\leq \theta(\|\sum_{i\in A_1}\pm
e_i\|+\|\sum_{i\in A_2}\pm e_i\|)\\
                               & \leq \theta C(\|\sum_{i\in
D_1}e_i\|+ \|\sum_{i\in D_2}e_i\|)\\
                                &\leq C^2(\|\sum_{i\in A_1}a_i
e_i\|+\|\sum_{i\in A_2}a_i e_i\|)\\ &\leq C^2(2K_b+1)\|\sum_{i\in
A}a_i e_i\|
\end{align*}
where $K_b$ is the basis constant.

Since  $\theta<\tau$ is arbitrary, so we can write
\begin{align}
\underset{\pm} max \|\sum_{i\in A}\pm e_i\|\leq
\frac{C^2(2K_b+1)}{\tau} \|\sum_{i\in A}a_i e_i\|.
\end{align}

Next we prove that $(e_i)_{i=1}^N$ is democratic. From Lemma~\ref{crd} it follows that it is sufficient to show that
$(e_i)_{i=1}^N$ is both conservative and reverse conservative.  Choose $A,B\subset\{1,\ldots,N\}$ with $A<B$ and
$|A|\leq |B|=k$. Consider $x=\theta 1_A+1_B$, where $0<\theta<\tau$.
Then $\mathcal{G}_{k}^\tau(x)=1_B$, so
\begin{align*}
\theta\|1_A\|\leq C \|1_B\|.
\end{align*}

Since  $\theta<\tau$ is arbitrary, so we get $(e_i)_{i=1}^N$ is
$\frac{C}{\tau}$ conservative. By similar reasoning,
 $(e_i)_{i=1}^N$ is $\frac{C}{\tau}$ reverse
conservative. Thus from Lemma~\ref{crd} it follows that
$(e_i)_{i=1}^N$ is democratic.
It follows from   Proposition~\ref{propp} and Lemma~\ref{lemp} that $(e_i)_{i=1}^N$ is quasi-greedy.  We have proved that $(e_i)_{i=1}^N$ is quasi-greedy and democratic,  and hence almost greedy by \cite{S}.
\end{proof}

To conclude this section we prove the converse of Theorems \ref{t1}.
in the  more general setting of the WTGA.
\begin{thm}\label{t3}
Let $(e_n)$ be almost greedy Markushevich basis of a Banach space
$X$. Let $0<\tau<1$ be weakness parameter and $0\leq \lambda <1$ be any scalar. Then exists a constant
$C$ such that for any $x\in X$, $A\subset \mathbb{N}$ with $|A\cap
\Lambda_m^{\tau}(x)|\leq \lambda m, ~|A|\leq m$ and $a_i\in
\mathbb{R}$, $i\in A$, we have
\begin{align*}
\|x-{G}_m^{\tau}(x)\| \leq ~C   \|x-\underset{i\in
A}{\sum}a_ie_i\|.
\end{align*}

\end{thm}

\begin{proof}
It follows from \cite{T1} that for any $x\in X$, $m\in \N$ there
exists a constant $C(\tau)$ such that
\begin{align}\label{ewg}
\|x-G_m^{\tau}(x)\|\leq C(\tau)\|x-G_m(x)\|.
\end{align}

The rest of the argument is very similar to Theorem~\ref{thm: agcharacterization}. Let $\phi(n)$ be the fundamental function for an  almost greedy Markushevich basis
$(e_n)$.   In the following inequalities the constants $C_1,C_2$ etc. depend only on $\lambda$ and  the quasi-greedy and democratic constants of $(e_n)$.
 Let $A \subset
\mathbb{N}$ with $|A\cap \Lambda_m^{\tau}(x)|\leq \lambda m$ and
$|A|\leq m$. Then for all coefficients $(a_i)_{i\in A}$,

\begin{align*}
\| x-\underset{i\in A}{\sum}a_ie_i \|& \geq C_1 \|G_{(1-\lambda)m}(x-\underset{i\in A}{\sum}a_ie_i)\|\\
                               &\geq C_2 \tau |e_{\rho(m)}^*(x)|\phi((1-\lambda)m)\\
                               &\geq C_2\tau(1-\lambda)|e_{\rho(m)}^*(x)|\phi(m)\\
                               &\geq C_3 \tau \|G_{2m}(x)-G_m(x)\|\\
                               &= C_3\tau\|(x-G_{m}(x))-(x-G_{2m}(x))\|.
\end{align*}
Here for the second inequality we used $(\ref{agmin})$ and the fact
that $|A\cap \Lambda_m^{\tau}|\leq \lambda m$, so the largest
$(1-\lambda)m$ coefficients  of $ x-\underset{i\in A}{\sum}a_ie_i$ are at least $\tau |e_{\rho(m)}^*(x)|$.

Now consider two cases, first suppose that $\|x-G_{2m}(x)\|\leq
\frac{1}{2} \|x-G_{m}(x)\|$. Then by the  triangle inequality
\begin{align*}
\| x-\underset{i\in A}{\sum}a_ie_i \| \geq
C_3\tau\|(x-G_{m}(x))-(x-G_{2m}(x))\| \geq
\frac{C_3\tau}{2}\|x-G_m(x)\|.
\end{align*}
Now suppose  $\|x-G_{2m}(x)\|>\frac{1}{2} \|x-G_{m}(x)\|$, then by
Theorem 3.3 \cite{S}, we have
\begin{align*}
\|x-G_m(x)\|\leq 2\|x-G_{2m}(x)\|\leq C_3\tau\sigma_m(x)\leq
C_3\tau\| x-\underset{i\in A}{\sum}a_ie_i \|.
\end{align*}

Thus in both the cases we get
\begin{align*}
\|x-G_m(x)\|\leq C_1(\tau) \| x-\underset{i\in A}{\sum}a_ie_i \|
\end{align*}
for some constant $C_1(\tau)$. Now the result follows from
(\ref{ewg}).
\end{proof}

\section*{Acknowledgements}
The second author would like to thank Gideon Schechtman for the
invaluable discussions and suggestions while preparation of this
paper.


\begin{thebibliography}{99}
\bibitem{ALq}F. Albiac and J.L. Ansorena, {\em Characterization of
1-quasi-greedy basis}, J. Approx. Theory {\bf 201} (2016), 7--12.

\bibitem{AL}F. Albiac and J.L. Ansorena, {\em Characterization of
1-almost greedy basis}, Rev. Mat. Complut. {\bf 30} (2017), 13--24.

\bibitem{AW}F. Albiac and P.Wojtaszczyk, {\em Characterization of
1-greedy basis}, J. Approx. Theory {\bf 138} (2006), no. 1, 65--86.

\bibitem{BB}P.M. Berna and O. Blasco, {\em Characterization of
greedy bases in Banach spaces}, J. Approx. Theory {\bf 215} (2017),
28--39.



\bibitem{DKK} S.J. Dilworth, N.J. Kalton and Denka Kutzarova, {\em On the existence of almost greedy bases in Banach spaces}, Studia Math. {\bf 159} (2003), 67--101.

\bibitem{DKSW}S.J. Dilworth,  Denka Kutzarova, T. Schlumprecht and P.Wojtaszczyk, {\em Weak thresholding greedy algorithms in Banach spaces}, J. Funct. Anal. {\bf 263} (2012), 3900--3921.


\bibitem{S}S.J. Dilworth, N.J. Kalton, D. Kutzarova and V.N. Temlyakov, {\em The thresholding greedy algorithm, greedy basis, and duality}, Constr. Approx.
{\bf 19} (2003), no. 4, 575--597.


\bibitem{KT}S.V. Konyagin and V.N.Temlyakov, {\em A remark on greedy
approximation in Banach spaces}, East. J. Approx. {\bf 5} (1999),
365--379.


\bibitem{T1}S.V. Konyagin and V.N.Temlyakov, {\em Greedy approximation with regard to bases and general minimal systems}, Serdics Math. J. (2002),
305--328.

\bibitem{T}V.N.Temlyakov, {\em The best $m$-term approximation and Greedy Algorithms}, Adv. Comput. Math. (1998),
249--265.



\bibitem{W}P.Wojtaszczyk, {\em Greedy algorithm for general
biorthogonal systems}, J. Approx. Theory {\bf 107} (2000), 293--314.

\end{thebibliography}
\end{document}